\documentclass[a4paper]{amsart}
\usepackage{graphicx}
\usepackage{amssymb}
\usepackage{amsmath}
\usepackage{amsthm}
\usepackage{amscd}
\usepackage{color}
\usepackage{colordvi}
\usepackage[all,2cell]{xy}

\UseAllTwocells \SilentMatrices
\newtheorem*{theorem*}{Theorem A}
\newtheorem*{theorem**}{Theorem B}
\newtheorem{theorem}{Theorem}[section]
\newtheorem{corollary}[theorem]{Corollary}
\newtheorem*{corollary*}{Corollary}
\newtheorem{lemma}[theorem]{Lemma}
\newtheorem*{lemma*}{Lemma}
\newtheorem{proposition}[theorem]{Proposition}
\newtheorem*{proposition*}{Proposition}
\theoremstyle{remark}
\newtheorem{remark}[theorem]{Remark}
\newtheorem*{remark*}{Remark}

\newtheorem*{example*}{Example}
\newtheorem*{observation*}{Observation}
\theoremstyle{definition}
\newtheorem{definition}[theorem]{Definition}
\newtheorem*{definition*}{Definition}

\newtheorem*{conjecture*}{Conjecture}

\numberwithin{equation}{section}

\begin{document}
\title[Models for chain homotopy category of relative acyclic complexes]{Models for chain homotopy category of relative acyclic complexes}
\author[Jiangsheng Hu, Wei Ren, Xiaoyan Yang, Hanyang You] {Jiangsheng Hu, Wei Ren, Xiaoyan Yang, Hanyang You$^{\dag}$}

\thanks{}
\date{\today}

\thanks{${}^{\dag}$ Corresponding author:  youhanyang@hznu.edu.cn }

\maketitle

\dedicatory{}%
\commby{}%
\begin{abstract}
Let $(\mathcal{X}, \mathcal{Y})$ be a balanced pair in an abelian category $\mathcal{A}$.
Denote by ${\bf K}_{\mathcal{E}\text{-}{\rm ac}}(\mathcal{X})$ the chain homotopy category of right $\mathcal{X}$-acyclic complexes with all items in $\mathcal{X}$, and dually by ${\bf K}_{\mathcal{E}\text{-}{\rm ac}}(\mathcal{Y})$ the chain homotopy category of left $\mathcal{Y}$-acyclic complexes with all items in $\mathcal{Y}$. We establish realizations of ${\bf K}_{\mathcal{E}\text{-}{\rm ac}}(\mathcal{X})$ and ${\bf K}_{\mathcal{E}\text{-}{\rm ac}}(\mathcal{Y})$ as homotopy categories of model categories under mild conditions. Consequently, we obtain relative versions of recollements of Krause and Neeman-Murfet. We further give applications to Gorenstein projective and Gorenstein injective modules.

\end{abstract}

\section{Introduction}

A model structure on a category $\mathcal{C}$ is a triple of three classes of morphisms, called cofibrations, fibrations, and weak equivalences, satisfying a few axioms; see \cite{Hov99,Qui67} for details. When $\mathcal{C}$ is an additive category equipped with a model structure, its homotopy category in the sense of Quillen (i.e., the localization of $\mathcal{C}$ with respect to weak equivalences) is a pretriangulated category in the sense of \cite{BR07}. If $\mathcal{C}$ is weakly idempotent complete, then the homotopy category of an exact model structure carries a triangulated structure (see \cite[Section 6]{Gil25}). Consequently, a Quillen equivalence between such model categories yields a triangle equivalence between their homotopy categories.

The Hovey correspondence \cite{Hov99,Hov02} of abelian categories is an effective tool for constructing model structures on abelian categories. It is inspired by the somewhat canonical model structure on a
Frobenius category, but with two cotorsion pairs mimicking the role played by the projectives and the injectives. Furthermore, Hovey's correspondence has been extended
as the one-one correspondence between exact model structures and the Hovey triples on weakly idempotent
complete exact categories, by Gillespie \cite{Gil11} (see also  \v{S}\v{t}ov\'{\i}\v{c}ek \cite{S14-2}).

An important application of model category theory lies in providing systematic methods for constructing recollements of triangulated categories (see \cite{Bec14,GC,Gil08,Gil16-2,HZZ25}). Recall that the notion of a recollement, introduced by Beilinson, Bernstein, and Deligne in \cite{BBD82}, can be viewed as a form of ``short exact sequence" of triangulated categories, in which the functors involved admit both left and right adjoints. For example, Becker \cite{Bec14} recovered Krause's recollement ${\bf K}_{\rm ac}(\mathcal{I})\longrightarrow {\bf K}(\mathcal{I})\longrightarrow {\bf D}(R)$ from \cite{H05}, and Gillespie \cite{Gil16-2} recovered Neeman-Murfet's recollement ${\bf K}_{\rm ac}(\mathcal{P})\longrightarrow {\bf K}(\mathcal{P})\longrightarrow {\bf D}(R)$ from \cite{Mur07} using the theory of abelian model categories. Here, ${\bf K}(\mathcal{I})$ (resp. ${\bf K}(\mathcal{P})$) denotes the chain homotopy category of all complexes of injective (resp. projective) modules, ${\bf K}_{\rm ac}(\mathcal{I})$
(resp. ${\bf K}_{\rm ac}(\mathcal{P})$) is the full subcategory of exact complexes of injective (resp. projective) modules, and ${\bf D}(R)$ is the derived category of a ring $R$.

Recall that a pair $(\mathcal{X}, \mathcal{Y})$ of additive subcategories in an abelian category $\mathcal{A}$ is said to be balanced if every object of $\mathcal{A}$ admits an $\mathcal{X}$-resolution that remains acyclic after applying $\operatorname{Hom}_{\mathcal{A}}(-,Y)$ for all $Y \in \mathcal{Y}$, and also admits a $\mathcal{Y}$-coresolution that is acyclic after applying $\operatorname{Hom}_{\mathcal{A}}(X,-)$ for all $X \in \mathcal{X}$. This condition implies a balancing phenomenon: the relative right-derived functors of $\operatorname{Hom}_\mathcal{A}(-,-)$ can be computed either via an $\mathcal{X}$-resolution of the first variable, or equivalently via a $\mathcal{Y}$-coresolution of the second variable. In other words, the Hom functor is right-balanced by $\mathcal{X} \times \mathcal{Y}$; see \cite[\S 8.2]{EJ00}. It is straightforward to verify that $(\mathcal{P},\mathcal{I})$ is a balanced pair. We refer to \cite{Chen10,EJ00,EPZ20} for more examples of balanced pairs.

Let $\mathcal{A}$ be an abelian category equipped with a balanced pair $(\mathcal{X}, \mathcal{Y})$. Denote by ${\bf K}(\mathcal{X})$ (resp. ${\bf K}(\mathcal{Y})$) the chain homotopy category of complexes with all items in $\mathcal{X}$ (resp. $\mathcal{Y}$), and by ${\bf K}_{\mathcal{E}\text{-}{\rm ac}}(\mathcal{X})$ (resp. ${\bf K}_{\mathcal{E}\text{-}{\rm ac}}(\mathcal{Y})$) the full subcategory of ${\bf K}(\mathcal{X})$ (resp. ${\bf K}(\mathcal{Y})$) consisting of complexes that are acyclic with respect to the functor ${\rm Hom}_\mathcal{A}(\mathcal{X},-)$ (resp. ${\rm Hom}_\mathcal{A}(-,\mathcal{Y})$). One then considers the following sequences of triangulated categories:
$${\bf K}_{\mathcal{E}\text{-}{\rm ac}}(\mathcal{X})\longrightarrow {\bf K}(\mathcal{X})\longrightarrow {\bf D}_\mathcal{X}(\mathcal{A})~\textrm{and}~{\bf K}_{\mathcal{E}\text{-}{\rm ac}}(\mathcal{Y})\longrightarrow {\bf K}(\mathcal{Y})\longrightarrow {\bf D}_\mathcal{Y}(\mathcal{A}),$$
where ${\bf D}_\mathcal{X}(\mathcal{A})$ (resp. ${\bf D}_\mathcal{Y}(\mathcal{A})$) is the relative derived category in the sense of \cite[Definition 3.1]{Chen10} (see also \cite{EJ00,LH15,SWSW08}). We have proved in \cite{HRYY25} that the chain homotopy categories ${\bf K}(\mathcal{X})$ and ${\bf K}(\mathcal{Y})$, and the relative derived categories ${\bf D}_\mathcal{X}(\mathcal{A})$ and ${\bf D}_\mathcal{Y}(\mathcal{A})$ can be realized as homotopy categories of model categories under certain conditions. This naturally leads us to seek realizations of ${\bf K}_{\mathcal{E}\text{-}{\rm ac}}(\mathcal{X})$ and ${\bf K}_{\mathcal{E}\text{-}{\rm ac}}(\mathcal{Y})$ as homotopy categories of suitable model categories, thereby obtaining relative versions of the Krause's and Neeman-Murfet's recollements.

We now outline the results of the paper. In Section \ref{preli}, we summarize some preliminaries and
basic facts which will be used throughout the paper.

In Section \ref{section3}, we realize the chain homotopy categories of complexes ${\bf K}_{\mathcal{E}\text{-}{\rm ac}}(\mathcal{X})$ and ${\bf K}_{\mathcal{E}\text{-}{\rm ac}}(\mathcal{Y})$ as homotopy categories of certain model categories. For the given balanced pair $(\mathcal{X}, \mathcal{Y})$, we denote by $\mathcal{E}$ the class of short exact sequences in $\mathcal{A}$ which remain exact by applying ${\rm Hom}_{\mathcal{A}}(X, -)$ for any $X\in \mathcal{X}$. It follows that $(\mathcal{A}, \mathcal{E})$ is an exact category.  Therefore, the category ${\rm Ch}(\mathcal{A}, \mathcal{E})$ of complexes over $(\mathcal{A}, \mathcal{E})$ with respect to the class ${\rm Ch}(\mathcal{E})$ of short exact sequences of complexes which are in $\mathcal{E}$ in each degree, is also an exact category (see \cite[Lemma 9.1]{Buh10}). By the Hovey correspondence between exact model structures and the Hovey triples on weakly idempotent
complete exact categories (see \cite{Gil11,S14-2}), we will denote the model structure $\mathcal{M}$ by the corresponding Hovey triples, and denote the homotopy cateogy of model categories by ${\rm Ho}(\mathcal{M})$. Under the assumption that $(\mathcal{E}\text{-}{\rm dw}\widetilde{\mathcal{X}})^{\perp}$ is closed under direct sums, we establish a hereditary model structure $\mathcal{M}_{ac\mathcal{X}} = (\mathcal{E}\text{-}{\rm ac}\widetilde{\mathcal{X}}, (\mathcal{E}\text{-}{\rm ac}\widetilde{\mathcal{X}})^{\perp}, {\rm Ch}(\mathcal{A}, \mathcal{E}))$ on the exact category ${\rm Ch}(\mathcal{A}, \mathcal{E})$ with a triangle equivalence ${\rm Ho}(\mathcal{M}_{ac\mathcal{X}}) \simeq {\bf K}_{\mathcal{E}\text{-}{\rm ac}}(\mathcal{X})$ (see Theorem \ref{thm:M4Sig}).

Dually, if ${}^{\perp}(\mathcal{E}\text{-}{\rm dw}\widetilde{\mathcal{Y}})$ is closed under direct products, then we have a hereditary model structure $\mathcal{M}_{ac\mathcal{Y}} = ({\rm Ch}(\mathcal{A}), {}^{\perp}(\mathcal{E}\text{-}{\rm ac}\widetilde{\mathcal{Y}}), \mathcal{E}\text{-}{\rm ac}\widetilde{\mathcal{Y}})$ on ${\rm Ch}(\mathcal{A}, \mathcal{E})$ with a triangle equivalence ${\rm Ho}(\mathcal{M}_{ac\mathcal{Y}}) \simeq {\bf K}_{\mathcal{E}\text{-}{\rm ac}}(\mathcal{Y})$ (see Remark \ref{rem:Sig}). In the specific case of $(\mathcal{X}, \mathcal{Y}) = (\mathcal{P}, \mathcal{I})$,  ${\bf K}_{\mathcal{E}\text{-}{\rm ac}}(\mathcal{Y})$ (resp. ${\bf K}_{\mathcal{E}\text{-}{\rm ac}}(\mathcal{X})$) is exactly the injective (resp. projective) stable derived category which have been studied by Gillespie in \cite{Gil16-2} and Krause in \cite{H05}.

In Section \ref{applications}, we obtain relative versions of Krause's and Neeman-Murfet's recollements. This is based on the models for ${\bf K}_{\mathcal{E}\text{-}{\rm ac}}(\mathcal{X})$ and ${\bf K}_{\mathcal{E}\text{-}{\rm ac}}(\mathcal{Y})$ in Section \ref{section3} and the models for ${\bf K}(\mathcal{X})$, ${\bf K}(\mathcal{Y})$, ${\bf D}_\mathcal{X}(\mathcal{A})$ and ${\bf D}_\mathcal{Y}(\mathcal{A})$ in \cite{HRYY25}. It is proved in Corollary \ref{cor:recollement} that if $(\mathcal{E}\text{-}{\rm dw}\widetilde{\mathcal{X}})^{\perp}$ is closed under direct sums, then there is a recollement:
$$\xymatrix@!=4pc{ {\bf K}_{\mathcal{E}\text{-}ac}(\mathcal{X}) \ar[r] & {\bf K}(\mathcal{X}) \ar@<-2ex>[l]
\ar@<+2ex>[l] \ar[r] & {\bf D}_\mathcal{X}(\mathcal{A}).
\ar@<-2ex>[l] \ar@<+2ex>[l]}$$
Dually, if $^{\perp}(\mathcal{E}\text{-}{\rm dw}\widetilde{\mathcal{Y}})$ is closed under direct products, then it is shown in Corollary  \ref{cor:recollement2} that there is a recollement:
$$\xymatrix@!=4pc{ {\bf K}_{\mathcal{E}\text{-}ac}(\mathcal{Y}) \ar[r] & {\bf K}(\mathcal{Y}) \ar@<-2ex>[l]
\ar@<+2ex>[l] \ar[r] & {\bf D}_\mathcal{Y}(\mathcal{A}).
\ar@<-2ex>[l] \ar@<+2ex>[l]}$$
These above reollements generalize the Krause's recollement in \cite[Theorem 7.7]{S14} and the Neeman-Murfet's recollement in \cite{Mur07,Nee08}. Denote by $\mathcal{GP}$ (resp. $\mathcal{GI}$) the subcategory which consisting of all Gorenstein projective (resp. injective) modules over a ring $R$. Let $R$ be a ring with finite Gorenstein weak dimension. It follows from \cite[Theorem 4.2]{WE24} and \cite[Lemma 5.7]{HRYY25} that $(\mathcal{GP},\mathcal{GI})$ is a balanced pair such that $(\mathcal{E}\text{-}{\rm dw}\widetilde{\mathcal{GP}})^{\perp}$ $=$ $^{\perp}(\mathcal{E}\text{-}{\rm dw}\widetilde{\mathcal{GI}})$ is closed under direct sums and direct products. In combination with this, we obtain recollements ${\bf K}_{\mathcal{E}\text{-}ac}(\mathcal{GP})\longrightarrow {\bf K}(\mathcal{GP})\longrightarrow {\bf D}_\mathcal{GP}(R)$ and ${\bf K}_{\mathcal{E}\text{-}ac}(\mathcal{GI})\longrightarrow {\bf K}(\mathcal{GI})\longrightarrow {\bf D}_\mathcal{GI}(R)$ (see Corollary \ref{cor:recollement3}). Here ${\bf D}_\mathcal{GP}(R)= {\bf D}_\mathcal{GI}(R)$ are called Gorenstein derived categories by Gao and Zhang in \cite{GZ10} (see Remark \ref{rmk:relativederivedcategory}). The principal technique we employ comes from the work of Becker \cite{Bec14} and Gillespie \cite{Gil08,Gil16-2}. They provided a method to construct recollement from three interrelated hereditary Hovey triples.

\section{Preliminaries}\label{preli}

Let $\mathcal{A}$ be a complete and cocomplete abelian category. A class of objects in $\mathcal{A}$ will be always assumed to be closed under isomorphisms and under finite direct sums. An {\em exact category} is a pair $(\mathcal{A}, \mathcal{E})$ where $\mathcal{E}$ is a class of  ``short exact sequences'' in $\mathcal{A}$, i.e. kernel-cokernel pairs $(i, p)$ depicted by $A'\stackrel{i}\rightarrowtail A\stackrel{p}\twoheadrightarrow A''$, satisfying some axioms; see Quillen's original definition in \cite{Qui73}. A map such as $i$ is called an {\em admissible monomorphism} while $p$ is called an {\em admissible epimorphism}. Recall that an exact category $(\mathcal{A}, \mathcal{E})$ is {\em weakly idempotent complete} if every split monomorphism has a cokernel and every split epimorphism has a kernel; see \cite[Definition 2.2]{Gil11} or \cite[Definition 7.2]{Buh10}. We refer to a readable exposition \cite{Buh10} for details on exact categories.

\subsection*{Cotorsion pairs}

In analogy to abelian categories, the axioms of exact categories allow for the usual construction of the Yoneda Ext bifunctor ${\rm Ext}^1_{\mathcal{E}}(M, N)$. It is the abelian group of equivalence classes of short exact sequences $N\rightarrowtail L\twoheadrightarrow M$. In particular, we get that ${\rm Ext}^1_{\mathcal{E}}(M, N) = 0$ if and only if every short exact sequence $N\rightarrowtail L\twoheadrightarrow M$ is isomorphic to the split exact sequence $N\rightarrowtail N\oplus M\twoheadrightarrow M$.

The definition of a cotorsion pair readily generalizes to exact categories; see \cite[Definition 2.1]{Gil11}. Specifically, a pair of classes $(\mathcal{F}, \mathcal{C})$ in $(\mathcal{A}, \mathcal{E})$ is a \emph{cotorsion pair} provided that $\mathcal{F} =  {^\perp}\mathcal{C}$ and $\mathcal{C} = \mathcal{F}^{\perp}$, where the left orthogonal class $^{\perp}\mathcal{C}$ consists of $F$ such that $\mathrm{Ext}^1_{\mathcal{E}}(F, X) = 0$ for all $X\in \mathcal{C}$, and the right orthogonal class $\mathcal{F}^{\perp}$ is defined similarly. We say the cotorsion pair $(\mathcal{F}, \mathcal{C})$ is {\em hereditary} if $\mathcal{F}$ is closed under taking kernels of admissible epimorphisms between objects of $\mathcal{F}$, and if $\mathcal{C}$ is closed under taking cokernels of admissible monomorphisms between objects of $\mathcal{C}$.

The cotorsion pair $(\mathcal{F}, \mathcal{C})$ is said to be {\em complete} if for any object $M\in \mathcal{A}$, there exist short exact sequences $C\rightarrowtail F \twoheadrightarrow M$ and $M\rightarrowtail C' \twoheadrightarrow F'$ with $F, F'\in \mathcal{F}$ and $C, C'\in \mathcal{C}$. In this case, $F \twoheadrightarrow M$ is called a \emph{special right $\mathcal{F}$-approximation} (or, \emph{special $\mathcal{F}$-precover}) of $M$, and $M\rightarrowtail C'$ is called  a \emph{special left $\mathcal{C}$-approximation} (or, \emph{special $\mathcal{C}$-preenvelope}) of $M$.

\subsection*{Approximation and balanced pairs}

Let $\mathcal{X}$ be a subcategory of the abelian category $\mathcal{A}$ and $M$ an object in $\mathcal{A}$. A morphism $f: X\rightarrow M$ (resp. $f: M\rightarrow X$) with $X\in \mathcal{X}$ is called a {\em right $\mathcal{X}$-approximation} (resp. {\em left $\mathcal{X}$-approximation}) of $M$, if any morphism from an object in $\mathcal{X}$ to $M$ (resp. $M$ to $\mathcal{X}$) factors through $f$. The subcategory $\mathcal{X}$ is called {\em contravariantly finite} (resp. {\em covariantly finite}) if each object in $\mathcal{A}$ has a right $\mathcal{X}$-approximation (resp. {\em left $\mathcal{X}$-approximation}).

Recall that a complex is  {\em right $\mathcal{X}$-acyclic} (resp. {\em left $\mathcal{Y}$-acyclic}) if it remains acyclic after applying ${\rm Hom}_\mathcal{A}(X, -)$ for all $X\in \mathcal{X}$ (resp. ${\rm Hom}_\mathcal{A}(-, Y)$ for all $Y\in \mathcal{Y}$).

A pair $(\mathcal{X}, \mathcal{Y})$ of subcategory in an abelian category $\mathcal{A}$ is called a {\em balanced pair} if the following conditions are satisfied (see \cite[Definition 1.1]{Chen10}):
\begin{enumerate}
\item the subcategory $\mathcal{X}$ is contravariantly finite and $\mathcal{Y}$ is covariantly finite;
\item for each object $M\in \mathcal{A}$, there is a complex $\cdots \rightarrow X_1\rightarrow X_0\rightarrow M\rightarrow 0$ with each $X_i\in \mathcal{X}$ which is both right $\mathcal{X}$-acyclic and left $\mathcal{Y}$-acyclic;
\item for each object $N\in \mathcal{A}$, there is a complex $0 \rightarrow M\rightarrow Y_0\rightarrow Y_1\rightarrow \cdots$ with each $Y_i\in \mathcal{Y}$ which is both right $\mathcal{X}$-acyclic and left $\mathcal{Y}$-acyclic;
\end{enumerate}

The balanced pair is called {\em admissible} if each right $\mathcal{X}$-approximation is an epimorphism and each left $\mathcal{Y}$-approximation is a monomorphism.
It follows from \cite[Proposition 2.6]{Chen10} that if there exist two complete and hereditary cotorsion pairs $(\mathcal{X}, \mathcal{Z})$ and $(\mathcal{Z}, \mathcal{Y})$ in $\mathcal{A}$, then the pair $(\mathcal{X}, \mathcal{Y})$ is an admissible balanced pair. In this case, $(\mathcal{X}, \mathcal{Z}, \mathcal{Y})$ is called a {\em cotorsion triple}. It follows from \cite[Theorem 4.4]{EPZ20} that the existence of complete and hereditary cotorsion triple in $\mathcal{A}$ is equivalent to that $\mathcal{A}$ has enough projective objects and injective objects.

\subsection*{Relative derived categories} 

Let $\mathcal{X}$ be a contravariantly finite subcategory of an abelian category $\mathcal{A}$. Denote by ${\bf K}(\mathcal{A})$ be the homotopy category of $\mathcal{A}$ and $\widetilde{\mathcal{E}}$ the subcategory of right $\mathcal{X}$-acyclic complexes, we recall the {\em relative derived category}  ${\bf D}_\mathcal{X}(\mathcal{A})$ of $\mathcal{A}$ with respect to $\mathcal{X}$ (see \cite[Definition 3.1]{Chen10}) is defined to be the Verdier quotient of ${\bf K}(\mathcal{A})$ modulo the subcategory consisting of objects in $\widetilde{\mathcal{E}}$, that is,
$${\bf D}_\mathcal{X}(\mathcal{A}):= {\bf K}(\mathcal{A})/ \widetilde{\mathcal{E}}.$$

\begin{remark}\label{rmk:relativederivedcategory}
Note that the derived category of exact category in the sense of \cite[Construction 1.5]{Nee90} is an example of relative derived category. In particular, if $\mathcal{X}$ is the full subcategory of Gorenstein projective objects in the sense of Enochs and Jenda in \cite{EJ00}, ${\bf D}_\mathcal{X}(\mathcal{A})$ is the Gorenstein derived category in the sense of Gao and Zhang in \cite{GZ10}.

\end{remark}

Dually, for a covariantly finite subcategory $\mathcal{Y}$, one can define the relative derived category ${\bf D}_\mathcal{Y}(\mathcal{A})$ of $\mathcal{A}$ with respect to $\mathcal{Y}$. Under the assumption that $(\mathcal{X}, \mathcal{Y})$ is a balanced pair, it follows from \cite[Proposition 2.2]{Chen10} that $\widetilde{\mathcal{E}}$ is exactly the complexes which is left $\mathcal{Y}$-acyclic, thus ${\bf D}_\mathcal{Y}(\mathcal{A})$ coincides with ${\bf D}_\mathcal{X}(\mathcal{A})$. Moreover, we have realized it as a homotopy category of a model structure, see \cite[Theorem 3.10]{HRYY25} for
details.

\subsection*{Hovey triples and model structures}

The notion of model structure is introduced by Quillen \cite{Qui67}, which refers to three specified classes of morphisms, called fibrations, cofibrations and weak equivalences,satisfing a few axions; see \cite{Qui67,Hov99} for details. A {\em model category} is a complete and cocomplete category equipped with a model structure.

Now suppose the exact category $(\mathcal{A}, \mathcal{E})$ has a model structure. An object $M\in \mathcal{A}$ is called {\em trivial} (resp. {\em cofibrant}, {\em fibrant}) if $0\rightarrowtail M$ (resp. $0\rightarrowtail M$, $M\twoheadrightarrow 0$) is a weak equivalence(resp. cofibration, fibration). We say $M$ is trivially cofibrant (resp. trivially fibrant) if it is both trivial and cofibrant (resp. fibrant). The subcategories of trivial, cofibrant and fibrant objects will be denoted by $\mathcal{A}_{tri}$, $\mathcal{A}_{c}$ and $\mathcal{A}_{f}$, respectively.

Recall that a {\em thick subcategory} means a class $\mathcal{W}$ of objects which is closed under direct summands, and such that if two out of three of the terms in a short exact sequence are in $\mathcal{W}$, then so is the third; see e.g. \cite[Definition 3.2]{Gil11}. Recall that a triple $(\mathcal{C}, \mathcal{W}, \mathcal{F})$ of subcategories in $(\mathcal{A}, \mathcal{E})$ is called a {\em (hereditary) Hovey triple}, if $\mathcal{W}$ is thick and both $(\mathcal{C}, \mathcal{W}\cap \mathcal{F})$ and $(\mathcal{C}\cap \mathcal{W}, \mathcal{F})$ are complete (hereditary) cotorsion pairs. It is well known that there is a correspondence between Hovey triples and model structures stated as follow:

\begin{lemma}\cite[Theorem 3.3]{Gil11}\label{lem:Gil3.3}
If the exact category $(\mathcal{A}, \mathcal{E})$ has a model structure admits a model structure, then the triple $(\mathcal{A}_{c}, \mathcal{A}_{tri}, \mathcal{A}_{f})$ of subcategories becomes a Hovey triple. If  $(\mathcal{A}, \mathcal{E})$ is weakly idempotent complete, then the converse holds. In this case, a map is a (trivial) cofibration if and only if it is an admissible monomorphism with a (trivially) cofibrant cokernel, and a map is a (trivial) fibration if and only if it is an admissible epimorphism with a (trivially) fibrant kernel. A map is weak equivalence if and only if it factors as a trivial cofibration followed by a trivial fibration.
\end{lemma}

Throughout this paper, we always denote a model structure by its corresponding Hovey triple $(\mathcal{A}_{c}, \mathcal{A}_{tri}, \mathcal{A}_{f})$.

Let $\mathcal{A}$ be a model category with a hereditary (that is, its corresponding Hovey triple is hereditary) model structure $\mathcal{M} = (\mathcal{A}_{c}, \mathcal{A}_{tri}, \mathcal{A}_{f})$. Its {\em homotopy category}, denote by $\mathrm{Ho}(\mathcal{M})$, is the localization of $\mathcal{A}$ with respect to the collection of weak equivalences. It is well known that $\mathcal{A}_{cf} = \mathcal{A}_{c}\cap\mathcal{A}_{f}$ is a Frobenius category, with $\omega = \mathcal{A}_{c}\cap\mathcal{A}_{tri}\cap\mathcal{A}_{f}$ being the class of projective-injective objects. Then the stable category $\underline{\mathcal{A}_{cf}} = \mathcal{A}_{cf}/\omega$ is a triangulated category. In this case one has a triangle equivalence $\mathrm{Ho}(\mathcal{A})\simeq \underline{\mathcal{A}_{cf}}$; see e.g. \cite[Theorem 1.3]{GLZ24}, \cite[Theorem 1.2.10]{Hov99}, \cite[Proposition 4.4]{Gil11} or \cite[Proposition 1.1.13]{Bec14}.

\subsection*{Exact category of complexes}
For a complex $\cdots\rightarrow C_{n+1}\stackrel{d_{n+1}}\rightarrow C_n\stackrel{d_n}\rightarrow C_{n-1}\rightarrow \cdots $ we denote ${\rm Ker}d_n$ by ${\rm Z}_nC$, ${\rm Im}d_{n+1}$ by ${\rm B}_nC$ and the $n$th homology ${\rm Z}_nC/{\rm B}_nC$ by ${\rm H}_nC$. For an object $A\in \mathcal{A}$, denote by ${\rm S}^nA$ the complex with $A$ in degree $n$ and all other entries 0, and ${\rm D}^nA$ the complex with $A$ in degree $n$ and $n-1$ and all other entries 0, with all maps 0 except $d_n = 1_A$. We refer to \cite[Lemma 3.1]{Gil04} and \cite[Lemma 4.2]{Gil11} for some useful isomorphisms with respect to complexes of the form ${\rm S}^nA$ and ${\rm D}^nA$. The suspension functor over complexes is denoted by $\Sigma$.

Given two complexes $C$ and $D$ and a chain map $f: C\rightarrow D$, denote by ${\rm Con}(f)$ the {\em mapping cone} of $f$. Recall that $f$ is {\em null homotopic}, denoted by $f\sim 0$, if there are maps $s_n: C_n\rightarrow D_{n+1}$ such that $f_n = d^D_{n+1}s_n + s_{n-1}d^C_n$. Chain maps $f, g: C\rightarrow D$ are called {\em chain homotopic}, denoted by $f\sim g$ if $f-g \sim 0$. In this sense $\{s_n\}$ are called a {\em chain homotopy}.

The {\em Hom-complex} ${\rm Hom}_{\mathcal{A}}(C, D)$ is defined with $n$th component
${\rm Hom}_{\mathcal{A}}(C, D)_{n} = \prod_{k\in \mathbb{Z}}{\rm Hom}_{\mathcal{A}}(X_{k}, Y_{k+n})$ and differential $(\delta_{n}f)_{k} = d_{k+n}^{D}f_{k} - (-1)^{n}f_{k-1}d_{k}^{C}$ for morphisms $f_{k}: C_{k}\rightarrow D_{k+n}$.

Let ${\rm Ch}(\mathcal{A}, \mathcal{E})$ be the exact category of chain complexes with respect to the class ${\rm Ch}(\mathcal{E})$ of short exact sequences of complexes which are in $\mathcal{E}$ degreewise.
Denote by ${\rm Ext}_{{\rm Ch}(\mathcal{A})}^1(C, D)$ the group of equivalence classes of short exact sequences $0\rightarrow D\rightarrow E\rightarrow C\rightarrow 0$ of complexes. Let ${\rm Ext}_{dw}^1(C, D)$ and ${\rm Ext}_{{\rm Ch}(\mathcal{E})}^1(C, D)$ be the subgroups of ${\rm Ext}_{{\rm Ch}(\mathcal{A})}^1(C, D)$ consisting of those short exact sequences which are in each degree split, and in $\mathcal{E}$ respectively. The following is well known; see e.g. \cite[Lemma 2.1]{Gil04}.

\begin{lemma}\label{lem:Gil04}
For chain complexes $C$ and $D$, one has
$${\rm Ext}_{dw}^1(C, \Sigma^{-n-1}D) \cong {\rm H}_n{\rm Hom}_\mathcal{A}(C, D) = {\rm Hom}_{{\rm Ch}(\mathcal{A})}(C, \Sigma^{-n}D)/\sim.$$
\end{lemma}


\section{Models for relative acyclic complexes}\label{section3}

Throughout the paper, let $\mathcal{A}$ be a complete abelian category which satisfies {\em AB5} (i.e. direct limits are exact in $\mathcal{A}$), and let $(\mathcal{X}, \mathcal{Y})$ be an admissible balanced pair in $\mathcal{A}$.

Recall that a complex $C$ is {\em right $\mathcal{X}$-acyclic} if it remains acyclic by applying ${\rm Hom}_{\mathcal{A}}(X, -)$ for any $X\in \mathcal{X}$, and dually, one has the notion of {\em left $\mathcal{Y}$-acyclic}; see \cite[pp. 2721]{Chen10}. We begin with the following observation, which will lead to \cite[Proposition 2.2]{Chen10} by a different and more straightforward proof.

\begin{lemma}\label{lem:XY-ac}
Let $0\rightarrow A\rightarrow B\rightarrow C\rightarrow 0$ be a short exact sequence. Then it is right $\mathcal{X}$-acyclic if and only if it is left $\mathcal{Y}$-acyclic.
\end{lemma}

In the following, $\mathcal{E}$ will denote the class of short exact sequences in $\mathcal{A}$ which are right $\mathcal{X}$-acyclic (equivalently, left $\mathcal{Y}$-acyclic). Then $(\mathcal{A}, \mathcal{E})$ is an exact category.

Inspired by \cite[Definition 3.3]{Gil04}, we have the following:

\begin{definition}\label{def:Com1}
\begin{enumerate}
\item $\widetilde{\mathcal{E}}$: the class of right $\mathcal{X}$-acyclic (left $\mathcal{Y}$-acyclic) complexes.
\item $\widetilde{\mathcal{X}}_{\mathcal{E}}$: the class of complexes $X\in\widetilde{\mathcal{E}}$ with all ${\rm Z}_nX \in \mathcal{X}$.
\item $\mathcal{E}\text{-}{\rm dw}\widetilde{\mathcal{X}}$: the class of complexes $X$ for which each item $X_n\in \mathcal{X}$.
\item $\mathcal{E}\text{-}{\rm dg}\widetilde{\mathcal{X}}$: the class of complexes $X\in \mathcal{E}\text{-}{\rm dw}\widetilde{\mathcal{X}}$ and for which every map $X\rightarrow E$ is null homotopic whenever $E\in \widetilde{\mathcal{E}}$.
\item $\mathcal{E}\text{-}{\rm ac}\widetilde{\mathcal{X}}$: $ = \mathcal{E}\text{-}{\rm dw}\widetilde{\mathcal{X}} \cap \widetilde{\mathcal{E}}$ the class of complexes $X$ which are right $\mathcal{X}$-acyclic with each item $X_n\in \mathcal{X}$.
\end{enumerate}
Dually, $\widetilde{\mathcal{Y}}_\mathcal{E}$, $\mathcal{E}\text{-}{\rm dw}\widetilde{\mathcal{Y}}$, $\mathcal{E}\text{-}{\rm dg}\widetilde{\mathcal{Y}}$ and $\mathcal{E}\text{-}{\rm ac}\widetilde{\mathcal{Y}}$ are defined.
\end{definition}

The prefix ``$\mathcal{E}$'' in ``$\mathcal{E}\text{-}{\rm dw}\widetilde{\mathcal{X}}$'' is used to indicate that we consider the right orthogonal $(\mathcal{E}\text{-}{\rm dw}\widetilde{\mathcal{X}})^{\perp}$  with respect to ${\rm Ext}^1_{{\rm Ch}(\mathcal{E})}(-, -)$.

It is direct to check the following facts.

\begin{lemma}\label{lem:SXDX}
\begin{enumerate}
\item For any $M\in \mathcal{X}$, one has ${\rm S}^nM\in\mathcal{E}\text{-}{\rm dg}\widetilde{\mathcal{X}}$, and ${\rm D}^nM\in \widetilde{\mathcal{X}}_\mathcal{E}$.
\item Let $0\rightarrow X'\rightarrow X''\rightarrow X\rightarrow 0$ be a short exact sequence in ${\rm Ch}(\mathcal{A}, \mathcal{E})$ with $X\in \mathcal{E}\text{-}{\rm dg}\widetilde{\mathcal{X}}$. Then $X'\in \mathcal{E}\text{-}{\rm dg}\widetilde{\mathcal{X}}$ if and only if $X''\in \mathcal{E}\text{-}{\rm dg}\widetilde{\mathcal{X}}$
\end{enumerate}
\end{lemma}

Let ${\bf K}(\mathcal{A})$ be the homotopy category of $\mathcal{A}$ and ${\bf K}(\mathcal{X})$ the subcategory of ${\bf K}(\mathcal{A})$ whose objects are complexes in $\mathcal{E}\text{-}{\rm dw}\widetilde{\mathcal{X}}$; see \cite[Proposition 3.5]{Chen10}.

Denote by ${\bf K}_{\mathcal{E}\text{-}{\rm ac}}(\mathcal{X})$ and ${\bf K}_{\mathcal{E}\text{-}{\rm ac}}(\mathcal{Y})$ the subcategory of complexes in $\mathcal{E}\text{-}{\rm ac}\widetilde{\mathcal{X}}$ and $\mathcal{E}\text{-}{\rm ac}\widetilde{\mathcal{Y}}$, respectively. In this section, we intend to find model structures to realize ${\bf K}_{\mathcal{E}\text{-}{\rm ac}}(\mathcal{X})$ and ${\bf K}_{\mathcal{E}\text{-}{\rm ac}}(\mathcal{Y})$; see Theorem \ref{thm:M4Sig} and Remark \ref{rem:Sig}. For this order we need the following results, which imply model structures for the chain homotopy categories and relative derived category, see \cite{HRYY25}.

\begin{lemma}\cite[Proposition 3.9]{HRYY25}\label{lem:ccpofdg}
There are complete cotorsion pairs $(\widetilde{\mathcal{X}}_\mathcal{E}, {\rm Ch}(\mathcal{A}))$ and $({\rm Ch}(\mathcal{A}), \widetilde{\mathcal{Y}}_\mathcal{E})$.
\end{lemma}

This result implies the following model structures:

\begin{lemma}\cite[Theorem 4.10]{HRYY25}\label{lem:modelofdw}
For the exact category ${\rm Ch}(\mathcal{A}, \mathcal{E})$;

If $(\mathcal{E}\text{-}{\rm dw}\widetilde{\mathcal{X}})^{\perp}$ is closed under direct sums, then $\mathcal{M}_{dw\mathcal{X}} = (\mathcal{E}\text{-}{\rm dw}\widetilde{\mathcal{X}}, (\mathcal{E}\text{-}{\rm dw}\widetilde{\mathcal{X}})^{\perp}, {\rm Ch}(\mathcal{A}))$ is a hereditary model structure  with ${\rm Ho}(\mathcal{M}_{dw\mathcal{X}}) \simeq {\bf K}(\mathcal{X})$;

If ${}^{\perp}(\mathcal{E}\text{-}{\rm dw}\widetilde{\mathcal{Y}})$ is closed under direct products, then $\mathcal{M}_{dw\mathcal{Y}} = ({\rm Ch}(\mathcal{A}), {^{\perp}(\mathcal{E}\text{-}{\rm dw}\widetilde{\mathcal{Y}})}, \mathcal{E}\text{-}{\rm dw}\widetilde{\mathcal{Y}})$ is a hereditary model structure with ${\rm Ho}(\mathcal{M}_{dw\mathcal{Y}}) \simeq {\bf K}(\mathcal{Y})$.

\end{lemma}

Furthermore, we obtain the following realization for the relative derived category ${\bf D}_\mathcal{X}(\mathcal{A})$,

\begin{lemma}\cite[Theorem 3.10]{HRYY25}\label{lem:modelofdg}
There are hereditary model structures $(\mathcal{E}\text{-}{\rm dg}\widetilde{\mathcal{X}}, \widetilde{\mathcal{E}}, {\rm Ch}(\mathcal{A}))$ and $({\rm Ch}(\mathcal{A}), \widetilde{\mathcal{E}}, \mathcal{E}\text{-}{\rm dg}\widetilde{\mathcal{Y}})$ on the exact category ${\rm Ch}(\mathcal{A}, \mathcal{E})$, with homotopy categories ${\bf D}_\mathcal{X}(\mathcal{A}) \simeq {\bf D}_\mathcal{Y}(\mathcal{A})$.
\end{lemma}

Note that ${\bf D}_\mathcal{X}(\mathcal{A})$ coincides with Neeman's derived category of the exact category $(\mathcal{A}, \mathcal{E})$ in \cite[Construction 1.5]{Nee90}.

In order to establish the model structure for ${\bf K}_{\mathcal{E}\text{-}{\rm ac}}(\mathcal{X})$, we need the following:

\begin{lemma}\label{lem:SnY}
For any complex $X\in \widetilde{\mathcal{E}}$ and any object $Y\in \mathcal{Y}$, the chain map $X\rightarrow {\rm S}^nY$ is null homotopic.
\end{lemma}

\begin{proof}
We infer from the chain map $f: X\rightarrow {\rm S}^nY$ that $f_nd_{n+1}^X = 0$. Then, $f_n: X_n\rightarrow Y$ induces a map
$g: X_n/{\rm B}_nX \cong {\rm Z}_{n-1}X \rightarrow Y$. Since $X\in \widetilde{\mathcal{E}}$, we infer from Lemma \ref{lem:XY-ac} that the short exact sequence $0\rightarrow {\rm Z}_{n-1}X\rightarrow X_{n-1}\rightarrow {\rm Z}_{n-2}X\rightarrow 0$ is left $\mathcal{Y}$-acyclic. Hence, ${\rm Hom}_{\mathcal{A}}(X_{n-1}, Y)\rightarrow {\rm Hom}_{\mathcal{A}}({\rm Z}_{n-1}X, Y)$ is epic, and then, there is a preimage of $g\in {\rm Hom}_{\mathcal{A}}({\rm Z}_{n-1}X, Y)$, i.e. a map
$s: X_{n-1}\rightarrow Y$, such that $f_n = sd_n^X$. Note that all $f_i$ other than $f_n$  are 0. Then, it follows that the chain map $f: X\rightarrow {\rm S}^nY$ is null homotopic.
\end{proof}

\begin{proposition}\label{prop:(acX,+)}
If $(\mathcal{E}\text{-}{\rm dw}\widetilde{\mathcal{X}})^{\perp}$ is closed under direct sums, then $(\mathcal{E}\text{-}{\rm ac}\widetilde{\mathcal{X}}, (\mathcal{E}\text{-}{\rm ac}\widetilde{\mathcal{X}})^{\perp})$ is a complete cotorsion pair in ${\rm Ch}(\mathcal{A}, \mathcal{E})$.
\end{proposition}

\begin{proof}
Let $\widehat{\mathcal{C}}$ be the collection of all complexes $C$ satisfying any chain map $X\rightarrow C$ from complexes $X\in \mathcal{E}\text{-}{\rm ac}\widetilde{\mathcal{X}}$ is null homotopic. Analogous to $${\rm Ext}^1_{{\rm Ch}(\mathcal{E})}(X, C) = {\rm Ext}^1_{dw}(X, C)\cong {\rm Hom}_{{\rm Ch}(\mathcal{A})}(X, \Sigma C)/\sim = 0$$ we can prove that $(\mathcal{E}\text{-}{\rm ac}\widetilde{\mathcal{X}})^{\perp} = \widehat{\mathcal{C}}$ and $\mathcal{E}\text{-}{\rm ac}\widetilde{\mathcal{X}} \subseteq {^{\perp}(\widehat{\mathcal{C}})} \subseteq \mathcal{E}\text{-}{\rm dw}\widetilde{\mathcal{X}}$.

Let $X\in {^{\perp}(\widehat{\mathcal{C}})}$. Let $Y$ be any object in $\mathcal{Y}$, and consider the short exact sequence $0\rightarrow {\rm S}^{n+1}Y\rightarrow {\rm D}^{n+1}Y\rightarrow {\rm S}^{n}Y\rightarrow 0$. It follows from Lemma \ref{lem:SnY} that  ${\rm S}^{n+1}Y \in \widehat{\mathcal{C}}$, and then any chain map
$X\rightarrow {\rm S}^{n}Y$ can be lifted to $X\rightarrow {\rm D}^{n+1}Y$. By the natural isomorphisms in \cite[Lemma 3.1]{Gil04}, we have the following commutative diagram
$$\xymatrix@C=5pt@R=20pt{ 0\ar[r] &{\rm Hom}_{{\rm Ch}(\mathcal{A})}(X, {\rm S}^nY) \ar[r] \ar[d]_{\cong} &{\rm Hom}_{{\rm Ch}(\mathcal{A})}(X, {\rm D}^{n+1}Y) \ar[r]\ar[d]^{\cong}&{\rm Hom}_{{\rm Ch}(\mathcal{A})}(X, {\rm S}^{n+1}Y)\ar[r]^{} \ar[d]^{\cong} &0\\
0\ar[r] &{\rm Hom}_{\mathcal{A}}(X_n/{\rm B}_nX, Y)\ar[r] &{\rm Hom}_{\mathcal{A}}(X_n, Y) \ar[r]
&{\rm Hom}_{\mathcal{A}}(X_{n+1}/{\rm B}_{n+1}X, Y) \ar[r] &0 }$$
Then, every sequence $0\rightarrow X_{n+1}/{\rm B}_{n+1}X\rightarrow X_n \rightarrow X_{n}/{\rm B}_{n}X\rightarrow 0$ is left $\mathcal{Y}$-acyclic. This yields that the complex $X$ is left $\mathcal{Y}$-acyclic, i.e. $X\in \widetilde{\mathcal{E}}$.
Hence, we have $X\in {^{\perp}(\widehat{\mathcal{C}})} \subseteq \mathcal{E}\text{-}{\rm dw}\widetilde{\mathcal{X}} \cap \widetilde{\mathcal{E}} = \mathcal{E}\text{-}{\rm ac}\widetilde{\mathcal{X}}$. This implies that $(\mathcal{E}\text{-}{\rm ac}\widetilde{\mathcal{X}}, (\mathcal{E}\text{-}{\rm ac}\widetilde{\mathcal{X}})^{\perp})$ is a cotorsion pair in ${\rm Ch}(\mathcal{A}, \mathcal{E})$.

Also from Lemma \ref{lem:modelofdg}, it follows that for any complex $C$, there is a short exact sequence $0\rightarrow Y\rightarrow E\rightarrow C\rightarrow 0$ in ${\rm Ch}(\mathcal{A}, \mathcal{E})$, for which $E\in \widetilde{\mathcal{E}}$ and $Y\in \mathcal{E}\text{-}{\rm dg}\widetilde{\mathcal{Y}}$. For $E$, by Lemma \ref{lem:modelofdw} we have a short exact sequence $0\rightarrow Z\rightarrow X\rightarrow E\rightarrow 0$ in ${\rm Ch}(\mathcal{A}, \mathcal{E})$, where $X\in \mathcal{E}\text{-}{\rm dw}\widetilde{\mathcal{X}}$ and $Z\in (\mathcal{E}\text{-}{\rm dw}\widetilde{\mathcal{X}})^{\perp}$. Consider the following pullback of $Y\rightarrow E$ and $X\rightarrow E$:
$$\xymatrix@C=20pt@R=20pt{ & 0\ar[d] & 0\ar[d] \\
 & Z \ar@{=}[r]^{} \ar[d]  &Z \ar[d]\\
0 \ar[r] & K \ar@{-->}[d] \ar@{-->}[r] & X \ar[r] \ar[d] & C \ar@{=}[d]  \ar[r] &0 \\
0 \ar[r] & Y \ar[r] \ar[d] & E \ar[r]^{} \ar[d] & C \ar[r] & 0\\
  & 0 & 0
  }$$
Since $Z\in (\mathcal{E}\text{-}{\rm dw}\widetilde{\mathcal{X}})^{\perp} \subseteq \widetilde{\mathcal{E}}$, we infer from the middle column that $X\in \widetilde{\mathcal{E}} \cap \mathcal{E}\text{-}{\rm dw}\widetilde{\mathcal{X}} = \mathcal{E}\text{-}{\rm ac}\widetilde{\mathcal{X}}$. Since $\mathcal{E}\text{-}{\rm dg}\widetilde{\mathcal{Y}} = (\widetilde{\mathcal{E}})^{\perp}$ and $\mathcal{E}\text{-}{\rm ac}\widetilde{\mathcal{X}} \subseteq \widetilde{\mathcal{E}}$, it follows that $Y\in \mathcal{E}\text{-}{\rm dg}\widetilde{\mathcal{Y}} \subseteq (\mathcal{E}\text{-}{\rm ac}\widetilde{\mathcal{X}})^{\perp}$. We infer from $\mathcal{E}\text{-}{\rm ac}\widetilde{\mathcal{X}}\subseteq \mathcal{E}\text{-}{\rm dw}\widetilde{\mathcal{X}}$ that
$Z\in (\mathcal{E}\text{-}{\rm dw}\widetilde{\mathcal{X}})^{\perp}\subseteq (\mathcal{E}\text{-}{\rm ac}\widetilde{\mathcal{X}})^{\perp}$. The left column then implies that $K\in (\mathcal{E}\text{-}{\rm ac}\widetilde{\mathcal{X}})^{\perp}$. Hence, for any complex $C$, we have constructed a short exact sequence $0\rightarrow K\rightarrow X\rightarrow C\rightarrow 0$ in ${\rm Ch}(\mathcal{A}, \mathcal{E})$, where $X\in \mathcal{E}\text{-}{\rm ac}\widetilde{\mathcal{X}}$ and $K\in (\mathcal{E}\text{-}{\rm ac}\widetilde{\mathcal{X}})^{\perp}$.

Furthermore, we will apply a standard argument (known as Salce's trick) to prove the another part. For any complex $C$, by Lemma \ref{lem:ccpofdg} there is a short exact sequence $0\rightarrow C\rightarrow Y\rightarrow L\rightarrow 0$ in ${\rm Ch}(\mathcal{A}, \mathcal{E})$, where $Y\in \widetilde{\mathcal{Y}}_\mathcal{E}$. For $L$ we have a short exact sequence $0\rightarrow K\rightarrow X\rightarrow L\rightarrow 0$ in ${\rm Ch}(\mathcal{A}, \mathcal{E})$, where $X\in \mathcal{E}\text{-}{\rm ac}\widetilde{\mathcal{X}}$ and $K\in (\mathcal{E}\text{-}{\rm ac}\widetilde{\mathcal{X}})^{\perp}$.  Consider the following pullback of $X\rightarrow L$ and $Y\rightarrow L$:

$$\xymatrix@C=20pt@R=20pt{ & & 0\ar[d] & 0\ar[d] \\
& & K \ar@{=}[r]^{} \ar[d]  &K \ar[d]\\
0 \ar[r] & C \ar@{=}[d] \ar[r] &D \ar@{-->}[r] \ar@{-->}[d] & X \ar[r] \ar[d] &0 \\
0 \ar[r] & C \ar[r] & Y \ar[r]^{} \ar[d] & L \ar[r] \ar[d]& 0\\
 & & 0 & 0
  }$$
Note that $Y\in \widetilde{\mathcal{Y}}_\mathcal{E} \subseteq \widehat{\mathcal{C}} = (\mathcal{E}\text{-}{\rm ac}\widetilde{\mathcal{X}})^{\perp}$ since the complexes in $\widetilde{\mathcal{Y}}_\mathcal{E}$ are contractible. Thus $0\rightarrow C\rightarrow D\rightarrow X\rightarrow 0$ is in $\mathcal{E}$ in each degree with $X\in \mathcal{E}\text{-}{\rm ac}\widetilde{\mathcal{X}}$ and $D\in (\mathcal{E}\text{-}{\rm ac}\widetilde{\mathcal{X}})^{\perp}$. This completes the proof.
\end{proof}

It is direct to check the following fact:

\begin{lemma}\label{lem:2Xe}

$\mathcal{E}\text{-}{\rm ac}\widetilde{\mathcal{X}} \cap (\mathcal{E}\text{-}{\rm ac}\widetilde{\mathcal{X}})^{\perp} = \widetilde{\mathcal{X}}_\mathcal{E}$.
\end{lemma}

\begin{theorem}\label{thm:M4Sig}
If $(\mathcal{E}\text{-}{\rm dw}\widetilde{\mathcal{X}})^{\perp}$ is closed under direct sums, then there exists a hereditary model structure $\mathcal{M}_{ac\mathcal{X}} = (\mathcal{E}\text{-}{\rm ac}\widetilde{\mathcal{X}}, (\mathcal{E}\text{-}{\rm ac}\widetilde{\mathcal{X}})^{\perp}, {\rm Ch}(\mathcal{A}))$ on the exact category ${\rm Ch}(\mathcal{A}, \mathcal{E})$ with a triangle equivalence $${\rm Ho}(\mathcal{M}_{ac\mathcal{X}}) \simeq {\bf K}_{\mathcal{E}\text{-}{\rm ac}}(\mathcal{X}).$$
\end{theorem}

\begin{proof}
We claim that $(\mathcal{E}\text{-}{\rm ac}\widetilde{\mathcal{X}})^{\perp}$ is a thick subcategory. First, we note that the cotorsion pair $(\mathcal{E}\text{-}{\rm ac}\widetilde{\mathcal{X}}, (\mathcal{E}\text{-}{\rm ac}\widetilde{\mathcal{X}})^{\perp})$ is hereditary. It suffices to prove that $(\mathcal{E}\text{-}{\rm ac}\widetilde{\mathcal{X}})^{\perp}$ is closed under taking kernels of admissible epimorphisms. That is, for any short exact sequence $0\rightarrow A\rightarrow B\rightarrow C\rightarrow 0$ in ${\rm Ch}(\mathcal{A}, \mathcal{E})$ for which $B, C\in (\mathcal{E}\text{-}{\rm ac}\widetilde{\mathcal{X}})^{\perp}$, we need to show that $A\in (\mathcal{E}\text{-}{\rm ac}\widetilde{\mathcal{X}})^{\perp}$.

It follows from Lemma \ref{prop:(acX,+)} that there is a short exact sequence $0\rightarrow A\rightarrow K\rightarrow Y\rightarrow 0$ in ${\rm Ch}(\mathcal{A}, \mathcal{E})$ with $K\in (\mathcal{E}\text{-}{\rm ac}\widetilde{\mathcal{X}})^{\perp}$ and $Y\in \mathcal{E}\text{-}{\rm ac}\widetilde{\mathcal{X}}$. Consider the pushout of $A\rightarrow B$ and $A\rightarrow K$ we have the following commutative diagram
$$\xymatrix@C=20pt@R=20pt{ & 0\ar[d] & 0\ar[d] \\
0 \ar[r] & A \ar[r] \ar[d]  &B \ar[r] \ar[d] & C \ar[r] \ar@{=}[d]  & 0\\
0 \ar[r] & K \ar@{-->}[d] \ar@{-->}[r] & D \ar[r] \ar[d] & C  \ar[r] &0 \\
0 \ar[r] & Y \ar@{=}[r] \ar[d] & Y \ar[d] & & \\
  & 0 & 0
  }$$
with $B,D\in (\mathcal{E}\text{-}{\rm ac}\widetilde{\mathcal{X}})^{\perp}$, thus $Y\in (\mathcal{E}\text{-}{\rm ac}\widetilde{\mathcal{X}})^{\perp}\cap \mathcal{E}\text{-}{\rm ac}\widetilde{\mathcal{X}} = \widetilde{\mathcal{X}}_\mathcal{E}$, and then $0\rightarrow A\rightarrow K\rightarrow Y\rightarrow 0$ is split degreewise. Since $Y$ is contractible, $A\rightarrow K$ is homotopically equivalent, this proves the above claim.

Then, by Lemma \ref{lem:ccpofdg}, \ref{lem:2Xe} and Proposition \ref{prop:(acX,+)}, together with the correspondence stated in Lemma \ref{lem:Gil3.3}, the model structure $\mathcal{M}_{ac\mathcal{X}}$ follows. The class of cofibrant-fibrant objects of the model structure is precisely $\mathcal{E}\text{-}{\rm ac}\widetilde{\mathcal{X}}$. Then we get the equivalence ${\rm Ho}(\mathcal{M}_{ac\mathcal{X}}) \simeq \mathcal{E}\text{-}{\rm ac}\widetilde{\mathcal{X}}/ \widetilde{\mathcal{X}}_\mathcal{E} \simeq {\bf K}_{\mathcal{E}\text{-}{\rm ac}}(\mathcal{X}).$

\end{proof}

Dually, we obtain the model structure $\mathcal{M}_{ac\mathcal{Y}}$ as follow:
\begin{remark}\label{rem:Sig}
If ${}^{\perp}(\mathcal{E}\text{-}{\rm dw}\widetilde{\mathcal{Y}})$ is closed under direct products, then there is a model structure $\mathcal{M}_{ac\mathcal{Y}} = ({\rm Ch}(\mathcal{A}), {}^{\perp}(\mathcal{E}\text{-}{\rm ac}\widetilde{\mathcal{Y}}), \mathcal{E}\text{-}{\rm ac}\widetilde{\mathcal{Y}})$ on the exact category ${\rm Ch}(\mathcal{A}, \mathcal{E})$ with a triangle equivalence $${\rm Ho}(\mathcal{M}_{ac\mathcal{Y}}) \simeq {\bf K}_{\mathcal{E}\text{-}{\rm ac}}(\mathcal{Y}).$$
\end{remark}

\begin{corollary}\label{cor:sing}
Assume that $(\mathcal{E}\text{-}{\rm dw}\widetilde{\mathcal{X}})^{\perp}$ and ${}^{\perp}(\mathcal{E}\text{-}{\rm dw}\widetilde{\mathcal{Y}})$ closed under direct sums and direct product, respectively. If $(\mathcal{E}\text{-}{\rm ac}\widetilde{\mathcal{X}})^{\perp} = {}^{\perp}(\mathcal{E}\text{-}{\rm ac}\widetilde{\mathcal{Y}})$, then there is a triangle-equivalence $${\bf K}_{\mathcal{E}\text{-}{\rm ac}}(\mathcal{X}) \simeq {\bf K}_{\mathcal{E}\text{-}{\rm ac}}(\mathcal{Y}).$$

\end{corollary}

\begin{proof}
Under the assumption, together with Theorem \ref{thm:M4Sig} and Remark \ref{rem:Sig}, we obtain model structures $\mathcal{M}_{ac\mathcal{X}}$ and $\mathcal{M}_{ac\mathcal{Y}}$ with common trivial objects. It follows from \cite[Corollary 1.4]{GLZ24} that there is a Quillen equivalence between the model categories $\mathcal{M}_{ac\mathcal{X}}$ and $\mathcal{M}_{ac\mathcal{Y}}$, which yields an equivalence of the corresponding homotopy categories ${\bf K}_{\mathcal{E}\text{-}{\rm ac}}(\mathcal{X}) \simeq {\rm Ho}(\mathcal{M}_{ac\mathcal{X}}) \simeq {\rm Ho}(\mathcal{M}_{ac\mathcal{Y}})\simeq {\bf K}_{\mathcal{E}\text{-}{\rm ac}}(\mathcal{Y})$ and completes the proof.
\end{proof}

\section{Applications}\label{applications}

Throughout this section, let $\mathcal{A}$ still be a complete abelian category which satisfies {\em AB5} with an admissible balanced pair $(\mathcal{X}, \mathcal{Y})$. The class of short exact sequences $\mathcal{E}$ given by right $\mathcal{X}$-acyclic as mentioned. We assume that $(\mathcal{E}\text{-}{\rm dw}\widetilde{\mathcal{X}})^{\perp}$ closed under direct sums and ${}^{\perp}(\mathcal{E}\text{-}{\rm dw}\widetilde{\mathcal{Y}})$ is closed under direct products on the exact category ${\rm Ch}(\mathcal{A}, \mathcal{E})$.

We recall the definition of recollement of triangulated categories, see \cite{BBD82}.

\begin{definition}
Let $\mathcal{T}_1$, $\mathcal{T}$ and $\mathcal{T}_2$ be
triangulated categories. A {\it recollement} of $\mathcal{T}$
relative to $\mathcal{T}_1$ and $\mathcal{T}_2$ is given by
$$\xymatrix@!=4pc{ \mathcal{T}_1 \ar[r]^{i_*=i_!} & \mathcal{T} \ar@<-3ex>[l]_{i^*}
\ar@<+3ex>[l]_{i^!} \ar[r]^{j^!=j^*} & \mathcal{T}_2
\ar@<-3ex>[l]_{j_!} \ar@<+3ex>[l]_{j_*}}$$ such that

(R1) $(i^*,i_*), (i_!,i^!), (j_!,j^!)$ and $(j^*,j_*)$ are adjoint
pairs of triangle functors;

(R2) $i_*$, $j_!$ and $j_*$ are full embeddings;

(R3) $j^!i_*=0$ (and thus also $i^!j_*=0$ and $i^*j_!=0$);

(R4) for each $X \in \mathcal {T}$, there are triangles

$$\begin{array}{l} j_!j^!X \rightarrow X  \rightarrow i_*i^*X  \rightarrow
\\ i_!i^!X \rightarrow X  \rightarrow j_*j^*X  \rightarrow
\end{array}$$ where the arrows to and from $X$ are the counits and the
units of the adjoint pairs respectively.

\end{definition}

Gillespie have obtained the following method to construct recollements.

\begin{lemma}\cite[Theorem 8.3]{Gil16-2}
Let $\mathcal{A}$ be an abelian category with three hereditary model structures
$$\mathcal{M}_1=(\mathcal{Q}_1,\mathcal{W}_1,\mathcal{R}),\ \ \mathcal{M}_2=(\mathcal{Q}_2,\mathcal{W}_2,\mathcal{R}),\ \ \mathcal{M}_3=(\mathcal{Q}_,\mathcal{W}_3,\mathcal{R})$$
with cores all coincide and $\mathcal{W}_3\cap \mathcal{Q}_1=\mathcal{Q}_2$ and $\mathcal{Q}_3\subseteq\mathcal{Q}_1$, then the sequence $${\rm Ho}(\mathcal{M}_2)\rightarrow{\rm Ho}(\mathcal{M}_1)\rightarrow{\rm Ho}(\mathcal{M}_3)$$
induces a recollement:
$$\xymatrix@!=4pc{ {\rm Ho}(\mathcal{M}_2) \ar[r] & {\rm Ho}(\mathcal{M}_1) \ar@<-2ex>[l]
\ar@<+2ex>[l] \ar[r] & {\rm Ho}(\mathcal{M}_3)
\ar@<-2ex>[l] \ar@<+2ex>[l]}$$
\end{lemma}

Combining Theorem \ref{thm:M4Sig}, Lemma \ref{lem:modelofdg} and Lemma \ref{lem:modelofdw}, we get three hereditary model structures on ${\rm Ch}(\mathcal{A}, \mathcal{E})$ as follow: $$\mathcal{M}_1=\mathcal{M}_{dw\mathcal{X}} = (\mathcal{E}\text{-}{\rm dw}\widetilde{\mathcal{X}}, (\mathcal{E}\text{-}{\rm dw}\widetilde{\mathcal{X}})^{\perp}, {\rm Ch}(\mathcal{A})),$$ $$\mathcal{M}_2=\mathcal{M}_{ac\mathcal{X}} = (\mathcal{E}\text{-}{\rm ac}\widetilde{\mathcal{X}}, (\mathcal{E}\text{-}{\rm ac}\widetilde{\mathcal{X}})^{\perp}, {\rm Ch}(\mathcal{A})),$$  $$\mathcal{M}_3=\mathcal{M}_{dg\mathcal{X}} = (\mathcal{E}\text{-dg}\widetilde{\mathcal{X}}, \widetilde{\mathcal{E}}, {\rm Ch}(\mathcal{A})),$$ whose cores are both $\widetilde{\mathcal{X}}_\mathcal{E}$. Since $\widetilde{\mathcal{E}} \cap  \mathcal{E}\text{-}{\rm dw}\widetilde{\mathcal{X}}= \mathcal{E}\text{-}{\rm ac}\widetilde{\mathcal{X}}$ and $\mathcal{E}\text{-dg}\widetilde{\mathcal{X}}\subseteq \mathcal{E}\text{-dw}\widetilde{\mathcal{X}}$, we have the following relative type of the Krause's recollement, compare to \cite[Theorem 7.7]{S14}.

\begin{corollary} \label{cor:recollement}
Let ${\bf K}_{\mathcal{E}\text{-}ac}(\mathcal{X})$, ${\bf K}(\mathcal{X})$ and ${\bf D}_\mathcal{X}(\mathcal{A})$ as mentioned above, there is an induced recollement:
$$\xymatrix@!=4pc{ {\bf K}_{\mathcal{E}\text{-}ac}(\mathcal{X}) \ar[r] & {\bf K}(\mathcal{X}) \ar@<-2ex>[l]
\ar@<+2ex>[l] \ar[r] & {\bf D}_\mathcal{X}(\mathcal{A})
\ar@<-2ex>[l] \ar@<+2ex>[l]}$$

\end{corollary}

Dually, together with \cite[Theorem 8.2]{Gil16-2} we obtain the following relative type of Neeman-Murfet's recollement, compare to \cite{Mur07,Nee08}:

\begin{corollary} \label{cor:recollement2}
Let ${\bf K}_{\mathcal{E}\text{-}ac}(\mathcal{Y})$, ${\bf K}(\mathcal{Y})$ and ${\bf D}_\mathcal{Y}(\mathcal{A})$ as mentioned above, there is an induced recollement:
$$\xymatrix@!=4pc{ {\bf K}_{\mathcal{E}\text{-}ac}(\mathcal{Y}) \ar[r] & {\bf K}(\mathcal{Y}) \ar@<-2ex>[l]
\ar@<+2ex>[l] \ar[r] & {\bf D}_\mathcal{Y}(\mathcal{A})
\ar@<-2ex>[l] \ar@<+2ex>[l]}$$

\end{corollary}

Let $R$ be an associative ring with identity. Recall a left $R$-modules $M$ is {\em Gorenstein projective} if $M \cong {\rm Z}_0 C$ for some totally $\mathcal{P}$-acyclic complex $C$, that is, $C$ is both right and left $\mathcal{P}$-acyclic with each item belongs to $\mathcal{P}$. Similarly, {\em Gorenstein injective} modules are defined. Denoted by $\mathcal{GP}$ (resp. $\mathcal{GI}$) the subcategory which consisting of all Gorenstein projective (resp. injective) modules over $R$, see \cite{EJ95} for details.

Recall that the {\em Gorenstein weak dimension} of $R$ is defined as to be the supremum of Gorenstein flat dimension of all left $R$-modules. Throughout this section, let $R$ be with finite Gorenstein weak dimension, and $\mathcal{A}$ be the category of left $R$-modules, it follows from \cite[Lemma 5.6 and Lemma 5.7]{HRYY25} that $(\mathcal{GP}, \mathcal{GI})$ is an admissible balanced pair and $(\mathcal{E}\text{-}{\rm dw}\widetilde{\mathcal{GP}})^{\perp} = {}^{\perp}(\mathcal{E}\text{-}{\rm dw}\widetilde{\mathcal{GI}})$. Together with Corollary \ref{cor:recollement} and \ref{cor:recollement2}, we get the following result.

\begin{corollary} \label{cor:recollement3}
Let $R$ be a ring with finite Gorenstein weak dimension, then we have recollements:
$$\xymatrix@!=4pc{ {\bf K}_{\mathcal{E}\text{-}ac}(\mathcal{GP}) \ar[r] & {\bf K}(\mathcal{GP}) \ar@<-2ex>[l]
\ar@<+2ex>[l] \ar[r] & {\bf D}_\mathcal{GP}(\mathcal{A})
\ar@<-2ex>[l] \ar@<+2ex>[l]}$$
and
$$\xymatrix@!=4pc{ {\bf K}_{\mathcal{E}\text{-}ac}(\mathcal{GI}) \ar[r] & {\bf K}(\mathcal{GI}) \ar@<-2ex>[l]
\ar@<+2ex>[l] \ar[r] & {\bf D}_\mathcal{GI}(\mathcal{A})
\ar@<-2ex>[l] \ar@<+2ex>[l]}$$
\end{corollary}

It is worth to note that ${\bf D}_\mathcal{GP}(\mathcal{A})$ coincides wtih ${\bf D}_\mathcal{GI}(\mathcal{A})$, which is exactly the Gorenstein derived category, see e.g. \cite{GZ10}.

\vskip 10pt

\noindent {\bf Acknowledgements.}\quad J.S. Hu is supported by the National Natural Science Foundation of China (Grant Nos. 12571035, 12171206) and Jiangsu 333 Project.
W. Ren is supported by the National Natural Science Foundation of China (No. 11871125), and Natural Science Foundation of Chongqing, China (No. cstc2018jcyjAX0541). 
X.Y. Yang is supported by the National Natural Science Foundation of China (Grant No. 12571035).  H.Y. You is supported by Zhejiang Provincial Natural Science Foundation of China (No. LQ23A010004) and the National Natural Science Foundation of China (Grant No. 12401043).

\bibliography{}

\vskip 10pt

\vspace{4mm}
\noindent\textbf{Jiangsheng Hu}\\
School of Mathematics, Hangzhou Normal University, Hangzhou 311121, P. R. China.\\
Email: \textsf{hujs@hznu.edu.cn}\\[1mm]
\textbf{Wei Ren}\\
School of Mathematical Sciences, Chongqing Normal University, Chongqing 401331, P. R. China\\
Email: \textsf{wren@cqnu.edu.cn}\\[1mm]
\textbf{Xiaoyan Yang}\\
Zhejiang University of Science and Technology, Hangzhou 310023, P. R. China.\\
Email: \textsf{yangxy@zust.edu.cn}\\[1mm]
\textbf{Hanyang You}\\
School of Mathematics, Hangzhou Normal University, Hangzhou 311121, P. R. China.\\
E-mail: \textsf{youhanyang@hznu.edu.cn}\\[1mm]

\end{document}